\documentclass[12pt,reqno]{amsart}
\usepackage{txfonts,fullpage}
\usepackage[dvipdfmx]{graphicx}

\newtheorem{thm}{Theorem}[section]
\newtheorem{prop}[thm]{Proposition}
\newtheorem{lem}[thm]{Lemma}
\newtheorem{cor}[thm]{Corollary}
\newtheorem{fact}[thm]{Fact}

\newtheorem{rem}[thm]{Remark}
\numberwithin{equation}{section}
\allowdisplaybreaks

\begin{document}

\title{An affine Weyl group action on the basic hypergeometric series arising from the $q$-Garnier system}
\date{}
\author{Taiki Idomoto}
\address{Graduate School of Science and Engineering, Kindai University, 3-4-1, Kowakae, Higashi-Osaka, Osaka 577-8502, Japan}
\author{Takao Suzuki}
\address{Department of Mathematics, Kindai University, 3-4-1, Kowakae, Higashi-Osaka, Osaka 577-8502, Japan}
\email{suzuki@math.kindai.ac.jp}

\maketitle

\begin{abstract}

Recently, we formulated the $q$-Garnier system and its variations as translations of an extended affine Weyl group of type $A^{(1)}_{2n+1}\times A^{(1)}_1\times A^{(1)}_1$.
On the other hand, those systems admit particular solutions in terms of the basic hypergeometric series ${}_{n+1}\phi_n$.
In this article, we investigate an action of the extended affine Weyl group on ${}_{n+1}\phi_n$.


\end{abstract}

\section{Introduction}

The Garnier system is an extension of the sixth Painlev\'e equation in several variables; see \cite{G,IKSY}.
Its $q$-difference analogue was proposed as the connection preserving deformation of a linear $q$-difference system in \cite{Sak1}.
Afterward, the $q$-Garnier system was investigated by the Pad\'e method in \cite{NY1,NY2}.
Then we obtained other directions of discrete time evolutions called variations.
In recent works \cite{OS1,OS2,Suz2}, we clarified that the $q$-Garnier system and all variations were given as elements called translations of an extended affine Weyl group of type $A^{(1)}_{2n+1}\times A^{(1)}_1\times A^{(1)}_1$.

The Garnier system admits a particular solution in terms of the Lauricella series $F_D$; see \cite{G,IKSY}.
Particular solutions of the $q$-Garnier system and one variation were investigated in \cite{NY1,Sak1,Suz1}.
As a result, two types of $q$-hypergeometric series appeared: the basic hypergeometric series ${}_{n+1}\phi_n$ and the $q$-Lauricella series $\phi^{(n)}_D$.
This fact presents two next problems.
One is particular solutions of the other variations and another is the action of the extended affine Weyl group on ${}_{n+1}\phi_n$.
In this article we give an answer to these problems.

The basic hypergeometric series is a power series
\[
	{}_{n+1}\phi_n\left[\begin{array}{c}a_1,\ldots,a_{n+1}\\b_1,\ldots,b_n\end{array};t\,\right] = \sum_{k=0}^{\infty}\frac{\prod_{l=1}^{n+1}(a_l;q)_k}{(q;q)_k\prod_{l=1}^{n}(b_l;q)_k}\,t^k,
\]
which converges for $|q|<1$ and $|t|<1$.
Here the symbol $(a;q)_k$ is the $q$-shifted factorial defined by
\[
	(a;q)_k = \prod_{l=1}^{k}(1-q^{l-1}a).
\]
It is known that ${}_{n+1}\phi_n$ satisfies various relations; see \cite{GR}.
For example, we have $q$-contiguity relations
\[
	\frac{1-a_l\,T_{q,t}}{1-a_l}\,{}_{n+1}\phi_n\left[\begin{array}{c}a_1,\ldots,a_{n+1}\\b_1,\ldots,b_n\end{array};t\,\right] = {}_{n+1}\phi_n\left[\begin{array}{c}a_1,\ldots,a_{l-1},q\,a_l,a_{l+1},\ldots,a_{n+1}\\b_1,\ldots,b_n\end{array};t\,\right],
\]
for $l=1,\ldots,n+1$,
\[
	\frac{1-q^{-1}b_l\,T_{q,t}}{1-q^{-1}b_l}\,{}_{n+1}\phi_n\left[\begin{array}{c}a_1,\ldots,a_{n+1}\\b_1,\ldots,b_n\end{array};t\,\right] = {}_{n+1}\phi_n\left[\begin{array}{c}a_1,\ldots,a_{n+1}\\b_1,\ldots,b_{l-1},q^{-1}b_l,b_{l+1},\ldots,b_n\end{array};t\,\right],
\]
for $l=1,\ldots,n$ and
\[
	\frac{\prod_{l=1}^{n}(1-b_l)}{t\prod_{l=1}^{n+1}(1-a_l)}\,(1-T_{q,t})\,{}_{n+1}\phi_n\left[\begin{array}{c}a_1,\ldots,a_{n+1}\\b_1,\ldots,b_n\end{array};t\,\right] = {}_{n+1}\phi_n\left[\begin{array}{c}q\,a_1,\ldots,q\,a_{n+1}\\q\,b_1,\ldots,q\,b_n\end{array};t\,\right],
\]
where $T_{q,t}$ is a $q$-shift operator satisfying $T_{q,t}\,x(t)=x(qt)$.
We also obtain an $(n+1)$-th order linear $q$-difference equation for $x={}_{n+1}\phi_n$
\[
	\left\{(1-T_{q,t})(1-q^{-1}b_1T_{q,t})\ldots(1-q^{-1}b_n\,T_{q,t})-t\,(1-a_1T_{q,t})\ldots(1-a_{n+1}T_{q,t})\right\}x = 0,
\]
from the $q$-contiguity relations.
In this article we derive these relations from the extended affine Weyl group systematically.

We state an outline of this article.
In Section \ref{Sec:q-Gar}, we recall the definition of the $q$-Garnier system and the derivation of its particular solution.
We first formulate a group of birational transformations on a field of rational functions, which is named $\widetilde{G}$.
This group is isomorphic to an extended affine Weyl group of type $A^{(1)}_{2n+1}\times A^{(1)}_1\times A^{(1)}_1$ and contains an abelian normal subgroup generated by translations.
In the following, we call this subgroup the $q$-Garnier system.
We next consider a direction of the $q$-Garnier system.
Via a certain specialization, it is reduced to a higher order generalization of the $q$-Riccati equation whose solution is described in terms of ${}_{n+1}\phi_n$.

The results of this article are given in the following sections.
In Section \ref{Sec:Par_Sol}, we find elements of $\widetilde{G}$ which is consistent with the specialization.
Those elements generate a subgroup of $\widetilde{G}$, which is named $\widetilde{F}$.
The group $\widetilde{F}$ is isomorphic to an extended affine Weyl group of type $A^{(1)}_n\times A^{(1)}_n$ and provides particular solutions of some directions of the $q$-Garnier system.
In Section \ref{Sec:Wey_HGF}, we formulate $(n+1)\times(n+1)$ matrices corresponding to the generators of $\widetilde{F}$.
Those matrices give a left action of $\widetilde{F}$ on an $(n+1)$-dimensional vector whose components are described in terms of ${}_{n+1}\phi_n$.
In Section \ref{Sec:q-HGE}, we formulate translations of $\widetilde{F}$ in a framework of $(n+1)\times(n+1)$ matrices.
Then we obtain the $q$-contiguity relations and the linear $q$-difference equation for ${}_{n+1}\phi_n$ again.

\begin{rem}
In this article we identify the $q$-Lauricella series
\[
	\phi^{(n)}_D\left[\begin{array}{c}a,b_1,\ldots,b_n\\c\end{array};t_1,\ldots,t_n\,\right] = \sum_{k_1,\ldots,k_n\geq0}\frac{(a;q)_{k_1+\ldots+k_n}\prod_{l=1}^{n}(b_l;q)_{k_l}}{(c;q)_{k_1+\ldots+k_n}\prod_{l=1}^{n}(q;q)_{k_l}}\prod_{l=1}^{n}a_l^{k_l},
\]
with ${}_{n+1}\phi_n$.
As a matter of fact, we have a transformation formula given in \cite{A}
\[
	\phi^{(n)}_D\left[\begin{array}{c}t,a_1^{-1}b_1,\ldots,a_n^{-1}b_n\\a_{n+1}\,t\end{array};a_1,\ldots,a_n\,\right] = \frac{(t;q)_{\infty}\prod_{l=1}^{n}(b_l;q)_{\infty}}{(a_{n+1}\,t;q)_{\infty}\prod_{l=1}^{n}(a_l;q)_{\infty}}\,{}_{n+1}\phi_n\left[\begin{array}{c}a_1,\ldots,a_{n+1}\\b_1,\ldots,b_n\end{array};t\,\right].
\]
\end{rem}

\begin{rem}
We formulated a birational representation of an extended affine Weyl group of type $A^{(1)}_{mn-1}\times A^{(1)}_{m-1}\times A^{(1)}_{m-1}$ as a generalization of the $q$-Garnier system in \cite{OS1,Suz2}.
We also investigated its $q$-hypergeometric solution for the case of $(m,n)=(3,2)$.
On the other hand, another generalization of the $q$-Garnier system was proposed together with its $q$-hypergeometric solution in \cite{P1,P2}.
We haven't clarified affine Weyl group symmetries of those $q$-hypergeometric solutions yet.
It's a future problem.
\end{rem}

\section{The $q$-Garnier system and the basic hypergeometric series}\label{Sec:q-Gar}

\subsection{The $q$-Garnier system}

Let $\varphi_{j,i}$ $(j\in\mathbb{Z}_{2n+2},i\in\mathbb{Z}_2)$ be dependent variables and $\alpha_j,\beta_i,\beta'_i$ $(j\in\mathbb{Z}_{2n+2},i\in\mathbb{Z}_2)$ parameters defined by
\[
	\alpha_j = \varphi_{j,0}\,\varphi_{j,1},\quad
	\beta_i = \prod_{j=0}^{2n+1}\varphi_{j,i},\quad
	\beta'_i = \prod_{j=0}^{2n+1}\varphi_{j,i+j}.
\]
We also set
\[
	q = \prod_{j=0}^{2n+1}\varphi_{j,0}\,\varphi_{j,1} = \prod_{j=0}^{2n+1}\alpha_j = \beta_0\,\beta_1 = \beta'_0\,\beta'_1.
\]

We first define birational transformations $r_j,s_i,s'_i$ $(j\in\mathbb{Z}_{2n+2},i\in\mathbb{Z}_2)$ by
\begin{align*}
	&r_j(\varphi_{j-1,i}) = \varphi_{j-1,i}\,\varphi_{j,i+1}\,\frac{1+\varphi_{j,i}}{1+\varphi_{j,i+1}},\quad
	r_j(\varphi_{j,i}) = \frac{1}{\varphi_{j,i+1}},\quad
	r_j(\varphi_{j+1,i}) = \varphi_{j,i}\,\varphi_{j+1,i}\,\frac{1+\varphi_{j,i+1}}{1+\varphi_{j,i}}, \\
	&r_j(\varphi_{k,i}) = \varphi_{k,i}\quad (k\neq j,j\pm1), \\
	&s_i(\varphi_{j,i}) = \frac{1}{\varphi_{j+1,i}}\frac{Q_{j,i}}{Q_{j+2,i}},\quad
	s_i(\varphi_{j,i+1}) = \varphi_{j,i}\,\varphi_{j,i+1}\,\varphi_{j+1,i}\,\frac{Q_{j+2,i}}{Q_{j,i}}, \\
	&s'_i(\varphi'_{j,i}) = \frac{1}{\varphi'_{j+1,i}}\frac{Q'_{j,i}}{Q'_{j+2,i}},\quad
	s'_i(\varphi'_{j,i+1}) = \varphi'_{j,i}\,\varphi'_{j,i+1}\,\varphi'_{j+1,i}\,\frac{Q'_{j+2,i}}{Q'_{j,i}},
\end{align*}
where
\[
	\varphi'_{j,i}=\varphi_{-j,i-j},\quad
	Q_{j,i} = \sum_{k=0}^{2n+1}\prod_{l=0}^{k-1}\varphi_{j+l,i},\quad
	Q'_{j,i} = \sum_{k=0}^{2n+1}\prod_{l=0}^{k-1}\varphi'_{j+l,i}.
\]
Then they act on the parameters as
\begin{align*}
	&r_j(\alpha_j) = \frac{1}{\alpha_j},\quad
	r_j(\alpha_{j\pm1}) = \alpha_{j\pm1}\,\alpha_j,\quad
	r_j(\alpha_k) = \alpha_k\quad (k\neq j,j\pm1),\quad
	r_j(\beta_i) = \beta_i,\quad
	r_j(\beta'_i) = \beta'_i, \\
	&s_i(\alpha_j) = \alpha_j,\quad
	s_i(\beta_i) = \frac{1}{\beta_i},\quad
	s_i(\beta_{i+1}) = \beta_{i+1}\,\beta_i^2,\quad
	s_i(\beta'_i) = \beta'_i,\quad
	s_i(\beta'_{i+1}) = \beta'_{i+1}, \\
	&s'_i(\alpha_j) = \alpha_j,\quad
	s'_i(\beta_i) = \beta_i,\quad
	s'_i(\beta_{i+1}) = \beta_{i+1},\quad
	s'_i(\beta'_i) = \frac{1}{\beta'_i},\quad
	s'_i(\beta'_{i+1}) = \beta'_{i+1}(\beta'_i)^2.
\end{align*}

\begin{fact}[{\cite{IIO,MOT,OS2}}]
If we set
\[
	G = \langle r_0,\ldots,r_{2n+1}\rangle,\quad
	H = \langle s_0,s_1\rangle,\quad
	H' = \langle s'_0,s'_1\rangle,
\]
then the groups $G,$ $H$ and $H'$ are isomorphic to the affine Weyl groups of type $A^{(1)}_{2n+1}$, $A^{(1)}_1$ and $A^{(1)}_1$ respectively.
Moreover, any two groups are mutually commutative.
\end{fact}

Recall that the affine Weyl group of type $A^{(1)}_{n-1}$ is generated by the generators $r_i$ $(i\in\mathbb{Z}_n)$ and the fundamental relations
\[
	r_0^2 = r_1^2 = 1,
\]
for $n=2$ and
\[
	r_i^2 = 1,\quad
	r_i\,r_j\,r_i = r_j\,r_i\,r_j\quad (|i-j|=1),\quad
	r_i\,r_j = r_j\,r_i\quad (|i-j|>1),
\]
for $n\geq3$.
Note that we interpret a composition of transformations in terms of automorphisms of the field of rational functions $\mathbb{C}(\varphi_{j,i})$.
For example, the compositions $r_0\,r_1$ act on the parameter $\alpha_0$ as
\[
	r_0\,r_1(\alpha_0) = r_0(\alpha_0\,\alpha_1) = r_0(\alpha_0)\,r_0(\alpha_1) = \alpha_1.
\]

We next define birational transformations $\pi,\pi',\rho$ by
\[
	\pi(\varphi_{j,i}) = \varphi_{j+1,i+1},\quad
	\pi'(\varphi_{j,i}) = \varphi_{j+1,i},\quad
	\rho(\varphi_{j,i}) = \varphi'_{j,i}.
\]
Then they act on the parameters as
\begin{align*}
	&\pi(\alpha_j) = \alpha_{j+1},\quad
	\pi(\beta_i) = \beta_{i+1},\quad
	\pi(\beta'_i) = \beta'_i, \\
	&\pi'(\alpha_j) = \alpha_{j+1},\quad
	\pi'(\beta_i) = \beta_i,\quad
	\pi'(\beta'_i) = \beta'_{i-1}, \\
	&\rho(\alpha_j) = \alpha_{-j},\quad
	\rho(\beta_i) = \beta'_i,\quad
	\rho(\beta'_i) = \beta_i.
\end{align*}

\begin{fact}[{\cite{OS2}}]
The transformations $\pi,\pi',\rho$ satisfy relations
\begin{align*}
	&\pi^{2n+2} = 1,\quad
	\pi^2 = (\pi')^2,\quad
	\pi\,\pi' = \pi'\pi,\quad
	\rho^2 = 1,\quad
	\pi\,\rho = \rho\,(\pi')^{-1}, \\
	&r_j\,\pi = \pi\,r_{j-1},\quad
	s_i\,\pi = \pi\,s_{i-1},\quad
	s'_i\,\pi = \pi\,s'_i,\quad
	r_j\,\pi' = \pi'r_{j-1},\quad
	s_i\,\pi' = \pi's_i,\quad
	s'_i\,\pi' = \pi's'_{i+1}, \\
	&r_j\,\rho = \rho\,r_{-j},\quad
	s_i\,\rho = \rho\,s'_i,
\end{align*}
for $j\in\mathbb{Z}_{2n+2}$ and $i\in\mathbb{Z}_2$.
\end{fact}

A semi-direct product $\widetilde{G}=\langle G,H,H'\rangle\rtimes\langle\pi,\pi',\rho\rangle$ can be regarded as an extended affine Weyl group of type $A_{2n+1}^{(1)}\times A_1^{(1)}\times A_1^{(1)}$.
Moreover, this group contains an abelian normal subgroup generated by translations; see \cite{OS2}.
In this article we call this subgroup the $q$-Garnier system.

\subsection{The basic hypergeometric series}

Consider a translation
\[
	\tau_c = (\pi')^{-1}\pi\,s'_0\,s_0.
\]
as a direction of the $q$-Garnier system.
It acts on the parameters as
\[
	\tau_c(\alpha_j) = \alpha_j,\quad
	\tau_c(\beta_0) = q^{-1}\beta_0,\quad
	\tau_c(\beta_1) = q\,\beta_1,\quad
	\tau_c(\beta'_0) = q^{-1}\beta'_0,\quad
	\tau_c(\beta'_1) = q\,\beta'_1.
\]

\begin{fact}[\cite{OS1,Suz1}]\label{Fac:q-Ric}
The action of $\tau_c$ on the dependent variables is consistent with a condition
\begin{equation}\label{Eq:Par_Sol}
	\varphi_{2j+1,0} = -\alpha_{2j+1},\quad
	\varphi_{2j+1,1} = -1\quad (j=0,\ldots,n),
\end{equation}
namely
\[
	\tau_c(\varphi_{2j+1,0})\bigm|_{Eq. \eqref{Eq:Par_Sol}} = -\alpha_{2j+1},\quad
	\tau_c(\varphi_{2j+1,1})\bigm|_{Eq. \eqref{Eq:Par_Sol}} = -1.
\]
Then we obtain a system of non-linear $q$-difference equations
\begin{equation}\label{Eq:q-Ric}
	\tau_c(\varphi_{2j,0}) = \frac{\varphi_{2j,0}}{\alpha_{2j-1}\,\alpha_{2j}}\frac{T_{j+1}}{T_j},\quad (j=0,\ldots,n),
\end{equation}
where
\[
	T_j = \prod_{l=0}^{n-1}\alpha_{2j+2l+1}\prod_{l=0}^{n}\alpha_{2l} + \sum_{k=1}^{n}(-1)^k(1-\alpha_{2j+2k-1})\prod_{l=k}^{n-1}\alpha_{2j+2l+1}\prod_{l=0}^{k-1}\varphi_{2l,0}\prod_{l=k}^{n}\alpha_{2l} + (-1)^n\prod_{l=0}^{n}\varphi_{2l,0}.
\]
\end{fact}

System \eqref{Eq:q-Ric} can be regarded as a higher order $q$-Riccati system.
It has $n+1$ dependent variables $\varphi_{0,0},\varphi_{2,0},\ldots,\varphi_{2n,0}$ and $2n+3$ parameters $\alpha_0,\alpha_1,\ldots,\alpha_{2n+1},\beta_0$ with a constraint
\[
	\prod_{j=0}^{n}\varphi_{2j,0} = \frac{(-1)^{n+1}\beta_0}{\prod_{j=0}^{n}\alpha_{2j+1}}.
\]
Note that condition \eqref{Eq:Par_Sol} implies that between parameters
\begin{equation}\label{Eq:Par_Sol_Par}
	\beta'_0 = \frac{\beta_0}{\prod_{j=0}^{n}\alpha_{2j+1}}.
\end{equation}

Under condition \eqref{Eq:Par_Sol_Par}, we define a $(2n+2)$-tuple of parameters $\mathbf{c}=\left(a_1,\ldots,a_{n+1},b_1,\ldots,b_n,c\right)$ by
\[
	a_j = \prod_{l=2n-2j+1}^{2n-1}\alpha_l,\quad
	b_j = \prod_{l=2n-2j}^{2n-1}\alpha_l,\quad
	c = \frac{\beta_0}{q\prod_{l=0}^{n}\alpha_{2l+1}}.
\]
We also use the notation
\[
	b_{n+1} = q,\quad
	a_{j+n+1} = q\,a_j,\quad
	b_{j+n+1} = q\,b_j,
\]
for a sake of convenience.
The original parameters are expressed in terms of $\mathbf{c}$ as
\[
	\alpha_{2j-2} = \frac{b_{n-j+1}}{a_{n-j+1}},\quad
	\alpha_{2j-1} = \frac{a_{n-j+1}}{b_{n-j}}\quad (j=0,\ldots,n),\quad
	\beta_0 = q\,c\prod_{l=1}^{n+1}\frac{a_l}{b_{l-1}}.
\]
The action of $\tau_c$ on $\mathbf{c}$ is given by
\[
	\tau_c(\mathbf{c}) = \left(a_1,\ldots,a_{n+1},b_1,\ldots,b_n,q^{-1}c\right).
\]
Let $\mathbf{x}=\mathbf{x}(\mathbf{c})$ be an $(n+1)$-dimensional vector defined by
\[
	\mathbf{x} = \begin{bmatrix}x_1\\\vdots\\x_{n+1}\end{bmatrix},\quad
	x_j = \prod_{l=1}^{n-j+1}(1-a_l)\prod_{l=n-j+2}^{n}(1-b_l)\,{}_{n+1}\phi_n\left[\begin{array}{c}q\,a_1,\ldots,q\,a_{n-j+1},a_{n-j+1},\ldots,a_{n+1}\\q\,b_1,\ldots,q\,b_{n-j+1},b_{n-j+2},\ldots,b_n\end{array};q\,c\,\right],
\]

\begin{fact}[\cite{OS1,Suz1}]\label{Fac:q-HGE}
The vector $\mathbf{x}$ satisfies a system of linear $q$-difference equations
\begin{equation}\label{Eq:q-HGE}
	\mathbf{x}(\tau_c(\mathbf{c})) = M_{\tau_c}\,\mathbf{x}(\mathbf{c}),
\end{equation}
with an $(n+1)\times(n+1)$ matrix
\begin{align*}
	(1-c)\,M_{\tau_c} &= \sum_{j=1}^{n+1}(b_{n-j+1}-a_{n-j+2}\,c)\,E_{j,j} \\
	&\quad + \sum_{j_1=1}^{n+1}\sum_{j_2=j_1+1}^{n+1}(b_{n-j_2+1}-a_{n-j_2+2})\,E_{j_1,j_2} + \sum_{j_1=1}^{n+1}\sum_{j_2=1}^{j_1-1}(b_{n-j_2+1}-a_{n-j_2+2})\,c\,E_{j_1,j_2},
\end{align*}
where $E_{j_1,j_2}$ is an $(n+1)\times(n+1)$ matrix with $1$ in $(j_1,j_2)$-th entry and $0$ elsewhere.
\end{fact}

System \eqref{Eq:q-HGE} provides a solution of system \eqref{Eq:q-Ric} in terms of ${}_{n+1}\phi_n$.

\begin{fact}[\cite{OS1,Suz1}]\label{Fac:q-HGE_Sol}
A solution of system \eqref{Eq:q-Ric} with the parameters $\mathbf{c}$ is given by
\begin{equation}\label{Eq:q-HGE_Sol}
	\varphi_{2j,0} = -\frac{x_{j+2}}{x_{j+1}}\quad (j=0,\ldots,n-1),\quad
	\varphi_{2n,0} = -q\,c\,\frac{x_1}{x_{n+1}}.
\end{equation}
\end{fact}

\begin{rem}
In a precise sense, the translation $\tau_c$ coincides with the $q$-Painlev\'e system $q$-$P_{(n+1,n+1)}$ given in \cite{Suz1}.
The original $q$-Garnier system given in \cite{Sak1} is derived from other translations of $\widetilde{G}$; see Remark \ref{Rem:q-Gar}.
\end{rem}

\section{A particular solution of the $q$-Garnier system}\label{Sec:Par_Sol}

\subsection{Results of this section}

Let us find elements of $\widetilde{G}$ such that condition \eqref{Eq:Par_Sol_Par} is invariant under their actions.
Then we obtain compositions of the transformations
\begin{equation}\begin{split}\label{Eq:Aff_Wey}
	&p_i = r_{-2i-3}\,r_{-2i-2}\,r_{-2i-3},\quad
	p'_i = r_{-2i-4}\,r_{-2i-3}\,r_{-2i-4}\quad (i\in\mathbb{Z}_{n+1}), \\
	&\sigma = \pi\,s_0\,r_1\,r_3\ldots r_{2n+1},\quad
	\sigma' = (\pi')^{-1}s'_0\,r_0\,r_2\ldots r_{2n},\quad
	\pi_1 = \pi^2,\quad
	\pi_2 = \rho\,(\pi')^{-1}\pi.
\end{split}\end{equation}

\begin{thm}\label{Thm:Par_Sol}
The actions of transformations given by \eqref{Eq:Aff_Wey} on the dependent variables $\varphi_{j,i}$ $(j\in\mathbb{Z}_{2n+2},i\in\mathbb{Z}_2)$ are consistent with condition \eqref{Eq:Par_Sol}.
\end{thm}

We will prove this theorem in Section \ref{SubSec:Par_Sol}.
Hence we obtain a subgroup of $\widetilde{G}$ which provides particular solutions of some directions of the $q$-Garnier system.
As a matter of fact, the translation $\tau_c=\sigma'\sigma$ is an element of this subgroup and Fact \ref{Fac:q-Ric} follows from Theorem \ref{Thm:Par_Sol}.

In the following, we always assume that condition \eqref{Eq:Par_Sol} holds.
Then the transformations given by \eqref{Eq:Aff_Wey} act on the dependent variables $\varphi_{2j,0}$ $(j\in\mathbb{Z}_{n+1})$ as
\begin{align*}
	p_i(\varphi_{-2i-4,0}) &= \frac{\varphi_{-2i-4,0}\left\{a_{i+1}-b_i+(a_{i+1}-a_i)\,\varphi_{-2i-2,0}\right\}}{a_i-b_i}, \\
	p_i(\varphi_{-2i-2,0}) &= \frac{(a_i-b_i)\,\varphi_{-2i-2,0}}{a_{i+1}-b_i+(a_{i+1}-a_i)\,\varphi_{-2i-2,0}}, \\
	p_i(\varphi_{2j,0}) &= \varphi_{2j,0}\quad (j\neq-2i-4,-2i-2),
\end{align*}
for $i=0,\ldots,n$,
\begin{align*}
	p'_i(\varphi_{-2i-4,0}) &= \frac{b_i-b_{i+1}+(b_i-a_{i+1})\,\varphi_{-2i-4,0}}{b_{i+1}-a_{i+1}}, \\
	p'_i(\varphi_{-2i-2,0}) &= \frac{(b_{i+1}-a_{i+1})\,\varphi_{-2i-4,0}\,\varphi_{-2i-2,0}}{b_i-b_{i+1}+(b_i-a_{i+1})\,\varphi_{-2i-4,0}}, \\
	p'_i(\varphi_{2j,0}) &= \varphi_{2j,0}\quad (j\neq-2i-4,-2i-2),
\end{align*}
for $i=0,\ldots,n$ and
\begin{align*}
	\sigma(\varphi_{2j,0}) &= \frac{(a_{n-j}-b_{n-j})\,\varphi_{2j,0}\,(b_{n-j-1}+a_{n-j-1}\,\varphi_{2j+2,0})}{(a_{n-j-1}-b_{n-j-1})(b_{n-j}+a_{n-j}\,\varphi_{2j,0})}, \\
	\sigma'(\varphi_{2j,0}) &= \frac{\varphi_{2j,0}\sum_{k=0}^{n}(-1)^k(b_{n-j-k-1}-a_{n-j-k})\prod_{l=1}^{k}\varphi_{2j+2l,0}}{\sum_{k=0}^{n}(-1)^k(b_{n-j-k}-a_{n-j-k+1})\prod_{l=1}^{k}\varphi_{2j+2l-2,0}}, \\
	\pi_1(\varphi_{2j,0}) &= \varphi_{2j+2,0}, \\
	\pi_2(\varphi_{2j,0}) &= \frac{b_{j-1}}{a_{j-1}\,\varphi_{-2j,0}}.
\end{align*}
They also act on the parameters $\mathbf{c}$ as
\[
	p_i(a_j) = a_{(i,i+1)(j)},\quad
	p_i(b_j) = b_j,\quad
	p_i(c) = c,
\]
for $i=0,\ldots,n$,
\begin{align*}
	&p'_0(a_j) = \frac{a_j}{b_1},\quad
	p'_0(b_j) = \frac{b_{(0,1)(j)}}{b_1},\quad
	p'_0(c) = c, \\
	&p'_i(a_j) = a_j,\quad
	p'_i(b_j) = b_{(i,i+1)(j)},\quad
	p'_i(c) = c, \\
	&p'_n(a_j) = \frac{q\,a_j}{b_n},\quad
	p'_n(b_j) = \frac{q\,b_{(n,n+1)(j)}}{b_n},\quad
	p'_n(c) = c,
\end{align*}
for $i=1,\ldots,n-1$ and
\begin{align*}
	&\sigma(a_j) = a_{j-1},\quad
	\sigma(b_j) = b_j,\quad
	\sigma(c) = c, \\
	&\sigma'(a_j) = a_{j+1},\quad
	\sigma'(b_j) = b_j,\quad
	\sigma'(c) = q^{-1}c, \\
	&\pi_1(a_j) = \frac{q\,a_{j-1}}{b_n},\quad
	\pi_1(b_j) = \frac{q\,b_{j-1}}{b_n},\quad
	\pi_1(c) = c, \\
	&\pi_2(a_j) = \frac{q\,a_{n-1}}{b_{2n-j}},\quad
	\pi_2(b_j) = \frac{q\,a_{n-1}}{a_{2n-j}},\quad
	\pi_2(c) = \frac{\prod_{l=1}^{n}b_l}{q\,c\prod_{l=1}^{n+1}a_l},
\end{align*}
where $(i_1,i_2)$ is a transposition.

\begin{prop}
If we set
\[
	F = \langle p_0,\ldots,p_n\rangle,\quad
	F' = \langle p'_0,\ldots,p'_n\rangle,
\]
then both $F$ and $F'$ are isomorphic to the affine Weyl group of type $A^{(1)}_n$ and $F\,F'=F'F$.
Moreover, the transformations $\sigma,\sigma',\pi_1,\pi_2$ satisfy relations
\begin{equation}\begin{split}\label{Eq:Dia_Aut_HGE}
	&p_i\,\sigma = \sigma\,p_{i+1},\quad
	p'_i\,\sigma = \sigma\,p'_i,\quad
	p_i\,\sigma' = \sigma'p_{i-1},\quad
	p'_i\,\sigma' = \sigma'p'_i, \\
	&p_i\,\pi_1 = \pi_1\,p_{i+1},\quad
	p'_i\,\pi_1 = \pi_1\,p'_{i+1},\quad
	p_i\,\pi_2 = \pi_2\,p'_{-i-3}, \\
	&\sigma\,\sigma' = \sigma'\sigma,\quad
	\sigma\,\pi_1 = \pi_1\,\sigma,\quad
	\sigma'\pi_1 = \pi_1\,\sigma',\quad
	\sigma\,\pi_2 = \pi_2\,\pi_1^{-1}(\sigma')^{-1}, \\
	&\pi_1^{n+1} = \pi_2^2 = 1, \quad
	\pi_1\,\pi_2 = \pi_2\,\pi_1^{-1}.
\end{split}\end{equation}
\end{prop}

\begin{proof}
Thanks to Theorem \ref{Thm:Par_Sol}, we can prove this proposition by using only the fundamental relations for $\widetilde{G}$.
For example, the relation $p_i\,\sigma=\sigma\,p_{i+1}$ of \eqref{Eq:Dia_Aut_HGE} is shown as
\begin{align*}
	p_i\,\sigma &= r_{-2i-3}\,r_{-2i-2}\,r_{-2i-3}\,\pi\,s_0\,r_1\,r_3\ldots r_{2n+1} \\
	&= \pi\,s_0\,r_{-2i-4}\,r_{-2i-3}\,r_{-2i-4}\,r_1\,r_3\ldots r_{2n+1} \\
	&= \pi\,s_0\,r_1\ldots r_{-2i-7}\,r_{-2i-1}\ldots r_{2n+1}\,r_{-2i-4}\,r_{-2i-3}\,r_{-2i-4}\,r_{-2i-3}\,r_{-2i-5} \\
	&= \pi\,s_0\,r_1\ldots r_{-2i-7}\,r_{-2i-1}\ldots r_{2n+1}\,r_{-2i-3}\,r_{-2i-4}\,r_{-2i-5} \\
	&= \pi\,s_0\,r_1\ldots r_{-2i-7}\,r_{-2i-1}\ldots r_{2n+1}\,r_{-2i-3}\,r_{-2i-5}\,r_{-2i-4}\,r_{-2i-5}\,r_{-2i-4} \\
	&= \pi\,s_0\,r_1\,r_3\ldots r_{2n+1}\,r_{-2i-5}\,r_{-2i-4}\,r_{-2i-5} \\
	&= \sigma\,p_{i+1}.
\end{align*}
We can show the other relations of \eqref{Eq:Dia_Aut_HGE} and the fundamental relations for $F$ and $F'$ in a similar way.
We don't state its detail here.
\end{proof}

Hence we can regard a semi-direct product $\widetilde{F}=\langle F,F'\rangle\rtimes\langle\sigma,\sigma',\pi_1,\pi_2\rangle$ as an extended affine Weyl group of type $A_n^{(1)}\times A_n^{(1)}$.

\subsection{Proof of Theorem \ref{Thm:Par_Sol}}\label{SubSec:Par_Sol}

We first prove for the transformation $\sigma=\pi\,s_0\,r_1\,r_3\ldots r_{2n+1}$.
The action of the composition of the transformations $r_1\,r_3\ldots r_{2n+1}$ on the dependent variables $\varphi_{2j+1,1}$ $(j=0,\ldots,n)$ is described as
\[
	r_1\,r_3\ldots r_{2n+1}(\varphi_{2j+1,1}) = \frac{1}{\varphi_{2j+1,0}}.
\]
Thus it is enough to show that
\begin{equation}\label{Eq:Par_Sol_Proof_1}
	\pi\,s_0(\varphi_{2j+1,0})\bigm|_{Eq. \eqref{Eq:Par_Sol}} = -1\quad (j=0,\ldots,n).
\end{equation}
Note that
\[
	\varphi_{2j+1,0} = \frac{\alpha_{2j+1}}{\varphi_{2j+1,1}}\quad (j=0,\ldots,n).
\]
From the definitions of $s_0$ and $\pi$, we obtain
\[
	\pi\,s_0(\varphi_{2j+1,0}) = \frac{1}{\varphi_{2j+3,1}}\frac{\sum_{k=0}^{2n+1}\prod_{l=0}^{k-1}\varphi_{2j+2+l,1}}{\sum_{k=0}^{2n+1}\prod_{l=0}^{k-1}\varphi_{2j+4+l,1}}.
\]
On the other hand, we have
\[
	\sum_{k=0}^{2n+1}\prod_{l=0}^{k-1}\varphi_{2j+l,1}\bigm|_{Eq. \eqref{Eq:Par_Sol}} = 1 + (-1)^n\prod_{l=0}^{n}\varphi_{2l,0},
\]
for any $j=0,\ldots,n$.
Combining them, we obtain equation \eqref{Eq:Par_Sol_Proof_1}.

We next prove for the transformation $\sigma'$.
The fundamental relations for $\widetilde{G}$ imply
\[
	\sigma' = r_1\,r_3\ldots r_{2n+1}\,(\pi')^{-1}s'_0.
\]
The action of $(\pi')^{-1}s'_0$ on the dependent variables is described as
\[
	(\pi')^{-1}s'_0(\varphi_{2j+1,1}) = \frac{1}{\varphi_{2j-1,0}}\frac{\sum_{k=0}^{2n+1}\prod_{l=0}^{k-1}\varphi_{2j-l,l+1}}{\sum_{k=0}^{2n+1}\prod_{l=0}^{k-1}\varphi_{2j-l-2,l+1}}.
\]
Thus it is enough to show that
\begin{equation}\label{Eq:Par_Sol_Proof_2}
	r_1\,r_3\ldots r_{2n+1}\left(\frac{1}{\varphi_{2j-1,0}}\frac{\sum_{k=0}^{2n+1}\prod_{l=0}^{k-1}\varphi_{2j-l,l+1}}{\sum_{k=0}^{2n+1}\prod_{l=0}^{k-1}\varphi_{2j-l-2,l+1}}\right)\bigm|_{Eq. \eqref{Eq:Par_Sol}} = -1\quad (j=0,\ldots,n).
\end{equation}
We focus on a factor
\[
	\sum_{k=0}^{2n+1}\prod_{l=0}^{k-1}\varphi_{2j-l,l+1} = 1 + \frac{\beta'_1}{\varphi_{2j+1,0}} + \sum_{k=1}^{n}(1+\varphi_{2j-2k+1,0})\prod_{l=0}^{k-2}\varphi_{2j-2l-1,0}\prod_{l=0}^{k-1}\varphi_{2j-2l,1}.
\]
By a direct calculation, we obtain
\begin{align*}
	&r_1\,r_3\ldots r_{2n+1}(1+\varphi_{2j-2k+1,0}) = \frac{1+\varphi_{2j-2k+1,1}}{\varphi_{2j-2k+1,1}}, \\
	&r_1\,r_3\ldots r_{2n+1}\left(\prod_{l=0}^{k-2}\varphi_{2j-2l-1,0}\right) = \prod_{l=0}^{k-2}\frac{1}{\varphi_{2j-2l-1,1}},\\
	&r_1\,r_3\ldots r_{2n+1}\left(\prod_{l=0}^{k-1}\varphi_{2j-2l,1}\right) = \frac{(1+\varphi_{2j+1,1})(1+\varphi_{2j-2k+1,0})}{(1+\varphi_{2j+1,0})(1+\varphi_{2j-2k+1,1})}\prod_{l=0}^{k-1}\varphi_{2j-2l-1,1}\,\varphi_{2j-2l,1}\,\varphi_{2j-2l+1,0}.
\end{align*}
It follows that
\begin{align*}
	&r_1\,r_3\ldots r_{2n+1}\left((1+\varphi_{2j-2k+1,0})\prod_{l=0}^{k-2}\varphi_{2j-2l-1,0}\prod_{l=0}^{k-1}\varphi_{2j-2l,1}\right)\bigm|_{Eq. \eqref{Eq:Par_Sol}} \\
	&= \frac{(1+\varphi_{2j+1,1})(1+\varphi_{2j-2k+1,0})}{1+\varphi_{2j+1,0}}\prod_{l=0}^{k-1}\varphi_{2j-2l,1}\,\varphi_{2j-2l+1,0}\bigm|_{Eq. \eqref{Eq:Par_Sol}} \\
	&= 0.
\end{align*}
Moreover, we have
\[
	r_1\,r_3\ldots r_{2n+1}\left(1+\frac{\beta'_1}{\varphi_{2j+1,0}}\right)\bigm|_{Eq. \eqref{Eq:Par_Sol}} = \left(1+\beta'_1\,\varphi_{2j+1,1}\right)\bigm|_{Eq. \eqref{Eq:Par_Sol}}
	= 1 - \beta'_1,
\]
and
\[
	r_1\,r_3\ldots r_{2n+1}\left(\frac{1}{\varphi_{2j-1,0}}\right)\bigm|_{Eq. \eqref{Eq:Par_Sol}} = \varphi_{2j-1,1}\bigm|_{Eq. \eqref{Eq:Par_Sol}} = -1.
\]
Combining them, we obtain equation \eqref{Eq:Par_Sol_Proof_2}.

We can prove for the other transformations in a similar way.
We don't state its detail here.

\section{An affine Weyl group action on the basic hypergeometric series}\label{Sec:Wey_HGF}

\subsection{Results of this section}

As is seen in the previous section, the translation $\tau_c$, or equivalently system \eqref{Eq:q-Ric}, is an element of the group $\widetilde{F}$.
Moreover, this group provides a symmetry for system \eqref{Eq:q-Ric}.

\begin{prop}\label{Prop:q-Ric_Sym}
Any element of the group $\langle F,F'\rangle\rtimes\langle\sigma,\sigma',\pi_1\rangle$ is commutative with the translation $\tau_c$.
Besides, the transformation $\pi_2$ satisfies a relation $\pi_2\,\tau_c=\tau_c^{-1}\pi_2$.
\end{prop}

\begin{proof}
The relation $\pi_2\,\tau_c=\tau_c^{-1}\pi_2$ follows from equation \eqref{Eq:Dia_Aut_HGE} as
\begin{align*}
	\pi_2\,\tau_c = \pi_2\,\sigma'\sigma
	= \sigma^{-1}\pi_2\,\sigma
	= \sigma^{-1}(\sigma')^{-1}\pi_2
	= \tau_c^{-1}\pi_2.
\end{align*}
Note that the relation $\sigma\,\pi_2=\pi_2\,(\sigma')^{-1}$ implies that $\pi_2\,\sigma'=\sigma^{-1}\pi_2$.
We can show the other commutation relations in a similar way.
We don't state its detail here.
\end{proof}

This proposition means that if $\varphi_{2j,0}$ $(j\in\mathbb{Z}_{n+1})$ satisfy system \eqref{Eq:q-Ric} with the parameters $\mathbf{c}$, then $\omega(\varphi_{2j,0})$ satisfy system \eqref{Eq:q-Ric} with the transformed parameters $\omega(\mathbf{c})$ for any $\omega\in\widetilde{F}$.
On the other hand, as is seen in Fact \ref{Fac:q-HGE_Sol}, the vector $\mathbf{x}$ provides a solution of system \eqref{Eq:q-Ric}.
Hence we expect that
\begin{equation}\label{Eq:HGE}
	-\frac{x_{j+2}(\omega(\mathbf{c}))}{x_{j+1}(\omega(\mathbf{c}))} = \omega(\varphi_{2j,0})\bigm|_{Eq. \eqref{Eq:q-HGE_Sol}}\quad (j=0,\ldots,n-1),
\end{equation}
for some $\omega\in\widetilde{F}$.
Then the right-hand side of \eqref{Eq:HGE} is a ratio of two linear combinations of the functions $x_1,\ldots,x_{n+1}$.
Based on this conjecture, we define $(n+1)\times(n+1)$ matrices corresponding to the generators of $\widetilde{F}$ by
\begin{align*}
	M_{p_0} &= \sum_{j=1}^{n}\frac{a_0-b_0}{a_1-b_0}\,E_{j,j} + E_{n+1,n+1} + \frac{a_0-a_1}{a_1-b_0}\,q\,c\,E_{n+1,1}, \\
	M_{p_i} &= \sum_{j=1}^{n-i}E_{j,j} + \frac{a_{i+1}-b_i}{a_i-b_i}\,E_{n-i+1,n-i+1} + \sum_{j=n-i+2}^{n+1}E_{j,j} + \frac{a_i-a_{i+1}}{a_i-b_i}\,E_{n-i+1,n-i+2},
\end{align*}
for $i=1,\ldots,n$,
\begin{align*}
	M_{p'_0} &= c^{\log_qb_1}\left( \sum_{j=1}^{n}E_{j,j} + \frac{b_0-a_1}{b_1-a_1}\,E_{n+1,n+1} + \frac{b_1-b_0}{b_1-a_1}\,E_{n+1,n} \right), \\
	M_{p'_i} &= \sum_{j=1}^{n-i}E_{j,j} + \frac{b_i-a_{i+1}}{b_{i+1}-a_{i+1}}\,E_{n-i+1,n-i+1} + \sum_{j=n-i+2}^{n+1}E_{j,j} + \frac{b_{i+1}-b_i}{b_{i+1}-a_{i+1}}\,E_{n-i+1,n-i}, \\
	M_{p'_n} &= c^{\log_qq^{-1}b_n}\left( \frac{b_n-a_{n+1}}{b_{n+1}-a_{n+1}}\,E_{1,1} + \sum_{j=2}^{n+1}E_{j,j} + \frac{b_{n+1}-b_n}{b_{n+1}-a_{n+1}}\,q^{-1}c^{-1}E_{1,n+1} \right),
\end{align*}
for $i=1,\ldots,n-1$ and
\begin{align*}
	M_{\sigma} &= \sum_{j=1}^{n+1}\frac{b_{n-j+1}\,(a_0-b_0)}{a_{n-j+1}-b_{n-j+1}}\,E_{j,j} - \sum_{j=1}^{n}\frac{a_{n-j+1}\,(a_0-b_0)}{a_{n-j+1}-b_{n-j+1}}\,E_{j,j+1} - a_{n+1}\,c\,E_{n+1,1}, \\
	M_{\sigma'} &= \frac{1}{1-c}\left( \sum_{j_1=1}^{n+1}\sum_{j_2=j_1}^{n+1}\frac{a_{n-j_2+2}-b_{n-j_2+1}}{a_1-b_0}\,E_{j_1,j_2} + \sum_{j_1=1}^{n+1}\sum_{j_2=1}^{j_1-1}\frac{a_{n-j_2+2}-b_{n-j_2+1}}{a_1-b_0}\,c\,E_{j_1,j_2} \right), \\
	M_{\pi_1} &= c^{\log_qq^{-1}b_n}\left( \sum_{j=1}^{n}E_{j,j+1} + q\,c\,E_{n+1,1} \right), \\
	M_{\pi_2} &= c^{\log_qq\,a_{n-1}}\left( E_{1,2} + \frac{b_n}{a_n}E_{2,1} + \sum_{j=3}^{n+1}\frac{b_n}{a_n\,a_{n+1}}\prod_{l=1}^{j-3}\frac{b_l}{a_l}\,c^{-1}E_{j,n+4-j} \right).
\end{align*}

\begin{thm}\label{Thm:q-HGE}
For any $\omega\in\langle F,p'_1,\ldots,p'_{n-1}\rangle\rtimes\langle\sigma,\sigma'\rangle$, the vector $\mathbf{x}$ satisfies
\begin{equation}\label{Eq:q-HGE_Wey}
	\mathbf{x}(\omega(\mathbf{c})) = M_{\omega}\,\mathbf{x}(\mathbf{c}).
\end{equation}
Here a composition of transformations is lifted to a product of matrices as
\[
	M_{\omega_2\,\omega_1} = \omega_2(M_{\omega_1})\,M_{\omega_2}.
\]
Besides, for any $\omega=p'_0,p'_n,\pi_1,\pi_2$, the vector $\omega^{-1}(M_{\omega})\,\mathbf{x}(\omega^{-1}(\mathbf{c}))$ is a solution of system \eqref{Eq:q-HGE} and two vectors $\mathbf{x}(\mathbf{c})$ and $\omega^{-1}(M_{\omega})\,\mathbf{x}(\omega^{-1}(\mathbf{c}))$ are linearly independent.
\end{thm}

We will prove this theorem in Section \ref{SubSec:q-HGE}.
In addition, we give a lemma which will be used to prove Theorem \ref{Thm:q-HGE_Lau}.

\begin{lem}\label{Lem:q-HGE_Lau_Thm}
The vector $\mathbf{x}$ satisfies
\[
	\mathbf{x}(p'_n\,\pi_1(\mathbf{c})) = \frac{1-a_0}{1-b_n}\,p'_n(M_{\pi_1})\,M_{p'_n}\,\mathbf{x}(\mathbf{c}).
\]
\end{lem}

This lemma can be proved in a similar way as Theorem \ref{Thm:q-HGE}.

\begin{rem}
Thanks to Theorem \ref{Thm:q-HGE}, we obtain the solution space of system \eqref{Eq:q-HGE} at $c=0$.
Besides, the vector $\pi_2^{-1}(M_{\pi_2})\,\mathbf{x}(\pi_2^{-1}(\mathbf{c}))$ is a solution of system \eqref{Eq:q-HGE} at $c=\infty$.
It also converges if we set $|b_1\ldots b_n|<|a_1\ldots a_{n+1}|$.
\end{rem}

\subsection{Proof of Theorem \ref{Thm:q-HGE}}\label{SubSec:q-HGE}

We first prove equation \eqref{Eq:q-HGE_Wey} for the transformation $\sigma'$.
Let $\mathbf{x}_k=\mathbf{x}_k(\mathbf{c})$ $(k\geq0)$ be $(n+1)$-dimensional vectors defined by
\[
	\mathbf{x}_k = \begin{bmatrix}x_{k,1}\\\vdots\\x_{k,n+1}\end{bmatrix},\quad
	x_{k,j} = q^k\prod_{l=1}^{n-j+1}(1-a_l)\prod_{l=n-j+2}^{n}(1-b_l)\prod_{l=1}^{n-j+1}\frac{(q\,a_l;q)_k}{(q\,b_l;q)_k}\prod_{l=n-j+2}^{n+1}\frac{(a_l;q)_k}{(b_l;q)_k}.
\]
Although the vectors $\mathbf{x}_k$ are independent of the parameter $c$, we use this notation for a sake of convenience.
Also let $M_{\sigma',0},M_{\sigma',1}$ be $(n+1)\times(n+1)$ matrices define by
\[
	M_{\sigma',0} = \sum_{j_1=1}^{n+1}\sum_{j_2=j_1}^{n+1}\frac{a_{n-j_2+2}-b_{n-j_2+1}}{a_1-b_0}\,E_{j_1,j_2},\quad
	M_{\sigma',1} = \sum_{j_1=1}^{n+1}\sum_{j_2=1}^{j_1-1}\frac{a_{n-j_2+2}-b_{n-j_2+1}}{a_1-b_0}\,E_{j_1,j_2}.
\]
Note that
\[
	\mathbf{x} = \sum_{k=0}^{\infty}\mathbf{x}_k\,c^k,\quad
	(1-c)\,M_{\sigma'} = M_{\sigma',0}+c\,M_{\sigma',1}.
\]
By using those notations, we can rewrite our goal $\mathbf{x}(\sigma'(\mathbf{c}))=M_{\sigma'}\,\mathbf{x}(\mathbf{c})$ into
\begin{align}
	&\mathbf{x}_0(\sigma'(\mathbf{c})) = M_{\sigma',0}\,\mathbf{x}_0(\sigma'(\mathbf{c})), \label{Eq:q-HGE_Wey_Proof_1} \\
	&q^{-k}\,\mathbf{x}_k(\sigma'(\mathbf{c})) = q^{-k+1}\,\mathbf{x}_{k-1}(\sigma'(\mathbf{c})) + M_{\sigma',0}\,\mathbf{x}_k(\mathbf{c}) + M_{\sigma',1}\,\mathbf{x}_{k-1}(\mathbf{c})\quad (k\geq1). \label{Eq:q-HGE_Wey_Proof_2}
\end{align}
Equation \eqref{Eq:q-HGE_Wey_Proof_1} is shown as follows.
The $j$-th row of the vector $M_{\sigma',0}\,\mathbf{x}_0(\sigma'(\mathbf{c}))$ is rewritten as
\begin{align*}
	&\sum_{j'=j}^{n+1}\frac{a_{n-j'+2}-b_{n-j'+1}}{a_1-b_0}\,x_{0,j'}(\mathbf{c}) \\
	&= \sum_{j'=j}^{n+1}\frac{a_{n-j'+2}-b_{n-j'+1}}{a_1-b_0}\prod_{l=1}^{n-j'+1}(1-a_l)\prod_{l=n-j'+2}^{n}(1-b_l) \\
	&= \sum_{j'=j}^{n}\prod_{l=1}^{n-j'+1}(1-a_{l+1})\prod_{l=n-j'+2}^{n}(1-b_l) - \sum_{j'=j}^{n}\prod_{l=1}^{n-j'}(1-a_{l+1})\prod_{l=n-j'+1}^{n}(1-b_l) + \prod_{l=1}^{n}(1-b_l) \\
	&= \prod_{l=1}^{n-j+1}(1-a_{l+1})\prod_{l=n-j+2}^{n}(1-b_l) \\
	&= x_{0,j}(\sigma'(\mathbf{c})).
\end{align*}
Equation \eqref{Eq:q-HGE_Wey_Proof_2} is shown as follows.
Its $j$-th row is described as
\[
	q^{-k}x_{k,j}(\sigma'(\mathbf{c})) = q^{-k+1}x_{k-1,j}(\sigma'(\mathbf{c})) + \sum_{j'=1}^{j-1}\frac{a_{n-j'+2}-b_{n-j'+1}}{a_1-b_0}\,x_{k-1,j'}(\mathbf{c}) + \sum_{j'=j}^{n+1}\frac{a_{n-j'+2}-b_{n-j'+1}}{a_1-b_0}\,x_{k,j'}(\mathbf{c}).
\]
Then we have
\[
	q^{-k+1}x_{k-1,j'+1}(\sigma'(\mathbf{c})) + \frac{a_{n-j'+2}-b_{n-j'+1}}{a_1-b_0}\,x_{k-1,j'}(\mathbf{c}) = q^{-k+1}x_{k-1,j'}(\sigma'(\mathbf{c}))\quad (j'=1,\ldots,j-1).
\]
It follows that
\[
	q^{-k+1}x_{k-1,j}(\sigma'(\mathbf{c})) + \sum_{j'=1}^{j-1}\frac{a_{n-j'+2}-b_{n-j'+1}}{a_1-b_0}\,x_{k-1,j'}(\mathbf{c}) = q^{-k+1}x_{k-1,1}(\sigma'(\mathbf{c})).
\]
Moreover, we have
\[
	q^{-k+1}x_{k-1,1}(\sigma'(\mathbf{c})) + x_{k,n+1}(\mathbf{c}) = q^{-k}x_{k,n+1}(\sigma'(\mathbf{c})),
\]
and
\[
	q^{-k}x_{k,j'+1}(\sigma'(\mathbf{c})) + \frac{a_{n-j'+2}-b_{n-j'+1}}{a_1-b_0}\,x_{k,j'}(\mathbf{c}) = q^{-k}x_{k,j'}(\sigma'(\mathbf{c}))\quad (j'=j,\ldots,n).
\]
Hence we obtain
\[
	q^{-k+1}x_{k-1,1}(\sigma'(\mathbf{c})) + \sum_{j'=j}^{n+1}\frac{a_{n-j'+2}-b_{n-j'+1}}{a_1-b_0}\,x_{k,j'}(\mathbf{c}) = q^{-k}x_{k,j}(\sigma'(\mathbf{c})).
\]

For the other transformations $p_0,\ldots,p_n,p'_1,\ldots,p'_{n-1},\sigma$, we can show in a similar way.
We don't state its detail here.

We next prove the latter half of the theorem for the transformation $\pi_2$.
We can rewrite our goal
\[
	\tau_c\,\pi_2^{-1}(M_{\pi_2})\,\mathbf{x}(\tau_c\,\pi_2^{-1}(\mathbf{c})) = M_{\tau_c}\,\pi_2^{-1}(M_{\pi_2})\,\mathbf{x}(\pi_2^{-1}(\mathbf{c})),
\]
into
\[
	\mathbf{x}(\tau_c^{-1}(\mathbf{c})) = \tau_c^{-1}(M_{\pi_2}^{-1})\,\pi_2(M_{\tau_c})\,M_{\pi_2}\,\mathbf{x}(\mathbf{c}),
\]
by using the fundamental relation $\pi_2^{-1}=\pi_2$ and Proposition \ref{Prop:q-Ric_Sym}.
Thus it is enough to show that
\[
	\tau_c^{-1}(M_{\tau_c}^{-1}) = \tau_c^{-1}(M_{\pi_2}^{-1})\,\pi_2(M_{\tau_c})\,M_{\pi_2}.
\]
Its right-hand side is expressed as
\begin{align*}
	(1-\pi_2(c))\,\tau_c^{-1}(M_{\pi_2}^{-1})\,\pi_2(M_{\tau_c})\,M_{\pi_2} &= \sum_{j=1}^{n+1}\left(\frac{1}{a_{n-j+2}}-\frac{\pi_2(c)}{b_{n-j+1}}\right)E_{j,j} \\
	&\quad + \sum_{j_2=1}^{n+1}\sum_{j_1=j_2+1}^{n+1}\left(\frac{1}{a_{n-j_2+2}}-\frac{1}{b_{n-j_2+1}}\right)\prod_{l=n-j_1+2}^{n-j_2+1}\frac{b_l}{a_l}\,E_{j_1,j_2} \\
	&\quad + \sum_{j_2=1}^{n+1}\sum_{j_1=1}^{j_2-1}\left(\frac{1}{a_{n-j_2+2}}-\frac{1}{b_{n-j_2+1}}\right)\prod_{l=n-j_2+2}^{n-j_1+1}\frac{a_l}{b_l}\,\pi_2(c)\,E_{j_1,j_2}.
\end{align*}
By using this expression, we can show
\[
	\tau_c^{-1}(M_{\tau_c})\,\tau_c^{-1}(M_{\pi_2}^{-1})\,\pi_2(M_{\tau_c})\,M_{\pi_2} = I.
\]

For the other transformations $\omega=p'_0,p'_n,\pi_1$, our goal is rewritten into
\[
	\omega(M_{\tau_c})\,M_{\omega} = \tau_c(M_{\omega})\,M_{\tau_c}.
\]
It can be shown by a direct calculation.
We don't state its detail here.

\section{$q$-Hypergeometric equations as translations}\label{Sec:q-HGE}

\subsection{Results of this section}

The group $\widetilde{F}$ contains three types of translations
\begin{align*}
	\tau_c &= \sigma'\sigma, \\
	\tau_i &= p_i\,p_{i+1}\ldots p_{i+n-1}\,\sigma\quad (i=1,\ldots,n+1), \\
	\tau_{i,j} &= p_i\,p_{i+1}\ldots p_{i+n-1}\,p'_j\,p'_{j+1}\ldots p'_n\,\pi_1\,p'_1\,p'_2\ldots p'_{j-1}\quad (i,j=1,\ldots,n+1),
\end{align*}
which act on the parameters $\mathbf{c}$ as
\begin{align*}
	&\tau_c(a_k) = a_k,\quad
	\tau_c(b_k) = b_k,\quad
	\tau_c(c) = q^{-1}c, \\
	&\tau_i(a_k) = q^{-\delta_{i,k}}a_k,\quad
	\tau_i(b_k) = b_k,\quad
	\tau_i(c) = c, \\
	&\tau_{i,j}(a_k) = q^{\delta_{j,n+1}-\delta_{i,k}}a_k,\quad
	\tau_{i,j}(b_k) = q^{\delta_{j,n+1}-\delta_{j,k}}b_k,\quad
	\tau_{i,j}(c) = c,
\end{align*}
where $\delta_{i,k}$ is the Kronecker's delta.
They generate an abelian normal subgroup $\widetilde{T}$ of $\widetilde{F}$ with relations
\begin{align*}
	&p_k\,\tau_i = \tau_{(k,k+1)(i)}\,p_k,\quad
	p'_k\,\tau_i = \tau_i\,p'_k, \\
	&\sigma\,\tau_i = \tau_{i-1}\,\sigma,\quad
	\sigma'\tau_i = \tau_{i+1}\,\sigma',\quad
	\pi_1\,\tau_i = \tau_{i-1}\,\pi_1,\quad
	\pi_2\,\tau_i = \tau_{i,-i-2}^{-1}\,\tau_i\,\tau_c^{-1}\pi_2,\quad \\
	&p_k\,\tau_{i,j} = \tau_{(k,k+1)(i),j}\,p_k,\quad
	p'_k\,\tau_{i,j} = \tau_{i,(k,k+1)(j)}\,p'_k, \\
	&\sigma\,\tau_{i,j} = \tau_{i-1,j}\,\sigma,\quad
	\sigma'\tau_{i,j} = \tau_{i+1,j}\,\sigma',\quad
	\pi_1\,\tau_{i,j} = \tau_{i-1,j-1}\,\pi_1,\quad
	\pi_2\,\tau_{i,j} = \tau_{-j-2,-i-2}^{-1}\,\pi_2,
\end{align*}
where $(k,k+1)(i)$ is the action of the permutation $(k,k+1)$ on the number $i$.
Recall that the translation $\tau_c$ is commutative with any element of $\langle F,F'\rangle\rtimes\langle\sigma,\sigma',\pi_1\rangle$ and satisfies $\pi_2\,\tau_c=\tau_c^{-1}\pi_2$; see Proposition \ref{Prop:q-Ric_Sym}.
Note that the extended affine Weyl group excepting $\pi_2$ is decomposed into semi-direct product of the group $\widetilde{T}$ and a finite Weyl group as
\[
	\langle F,F'\rangle\rtimes\langle\sigma,\sigma',\pi_1\rangle = \widetilde{T}\rtimes\langle p_1,\ldots,p_n,p'_1,\ldots,p'_n\rangle
\]

Among elements of $\widetilde{T}$, the translation $\tau_c$ was investigated in the previous sections.
As a matter of fact, the corresponding matrix is described as
\[
	M_{\tau_c} = \sigma'(M_{\sigma})\,M_{\sigma'},
\]
and Fact \ref{Fac:q-HGE} follows from Theorem \ref{Thm:q-HGE}.
Thus we investigate the other translations and formulate their corresponding matrices in this section.

Let $M_{\tau_i}$ $(i=1,\ldots,n+1)$ and $M_{\tau_{i,j}}$ $(i,j=1,\ldots,n+1)$ be $(n+1)\times(n+1)$ matrices defined by
\[
	M_{\tau_i} = p_i\ldots p_{i+n-1}(M_{\sigma})\times p_i\ldots p_{i+n-2}(M_{p_{i+n-1}})\times\ldots\times p_i(M_{p_{i+1}})\times M_{p_i},
\]
and
\begin{align*}
	\Delta_{i,j}\,M_{\tau_{i,j}} &= p_i\ldots p_{i+n-1}\,p'_j\ldots p'_n\,\pi_1\,p'_1\ldots p'_{j-2}(M_{p'_{j-1}})\times\ldots\times p_i\ldots p_{i+n-1}\,p'_j\ldots p'_n\,\pi_1(M_{p'_1}) \\
	&\quad\times p_i\ldots p_{i+n-1}\,p'_j\ldots p'_n(M_{\pi_1})\times p_i\ldots p_{i+n-1}\,p'_j\ldots p'_{n-1}(M_{p'_n})\times\ldots\times p_i\ldots p_{i+n-1}(M_{p'_j}) \\
	&\quad\times p_i\ldots p_{i+n-2}(M_{p_{i+n-1}})\times\ldots\times p_i(M_{p_{i+1}})\times M_{p_i},
\end{align*}
where
\begin{align*}
	&\Delta_{i,j} = \frac{1-b_j}{1-a_1}\quad (i\neq1,\ j\neq n+1),\quad
	\Delta_{1,j} = \frac{1-b_j}{1-q^{-1}a_1}\quad (j\neq n+1), \\
	&\Delta_{i,n+1} = q\,\frac{1-q^{-1}a_i}{1-a_i}\prod_{l=1}^{n}\frac{1-a_{l+1}}{1-q\,b_l}\quad (i\neq1),\quad
	\Delta_{1,n+1} = q\prod_{l=1}^{n}\frac{1-a_{l+1}}{1-q\,b_l}.
\end{align*}

\begin{thm}\label{Thm:q-HGE_Lau}
The vector $\mathbf{x}$ satisfies systems of $q$-difference equations
\begin{align}
	&\mathbf{x}(\tau_i(\mathbf{c})) = M_{\tau_i}\,\mathbf{x}(\mathbf{c})\quad (i=1,\ldots,n+1), \label{Eq:q-HGE_Lau1}\\
	&\mathbf{x}(\tau_{i,j}(\mathbf{c})) = M_{\tau_{i,j}}\,\mathbf{x}(\mathbf{c})\quad (i,j=1,\ldots,n+1). \label{Eq:q-HGE_Lau2}
\end{align}
\end{thm}

We will prove this theorem in Section \ref{SubSec:q-HGE_Lau}.
Systems \eqref{Eq:q-HGE_Lau1} and \eqref{Eq:q-HGE_Lau2} are equivalent to the $q$-contiguity relations for ${}_{n+1}\phi_n$.
Besides, system \eqref{Eq:q-HGE_Lau2} can be regarded as a $q$-analogue of the linear Pfaff system whose solution is described in terms of $F_D$.

\begin{cor}\label{Cor:q-HGE_Lau}
Let $\mathbf{y}=\mathbf{y}(\mathbf{c})$ be an $(n+1)$-dimensional vector defined by
\[
	\mathbf{y} = \begin{bmatrix}y_1\\\vdots\\y_{n+1}\end{bmatrix},\quad
	y_j = \phi^{(n)}_D\left[\begin{array}{c}q\,c,a_1^{-1}b_1,\ldots,a_n^{-1}b_n\\q\,a_{n+1}\,c\end{array};q\,a_1,\ldots,q\,a_{n-j+1},a_{n-j+2},\ldots,a_n\,\right].
\]
Then the vector $\mathbf{y}$ satisfies a system of $q$-difference equations
\begin{align*}
	\mathbf{y}(\tau_{i,j}(\mathbf{c})) &= \frac{1-b_j}{1-q^{-1}a_i}\,M_{\tau_{i,j}}\,\mathbf{y}(\mathbf{c})\quad (i,j=1,\ldots,n), \\
	\mathbf{y}(\tau_{n+1,j}(\mathbf{c})) &= \frac{1-b_j}{1-a_{n+1}\,c}\,M_{\tau_{n+1,j}}\,\mathbf{y}(\mathbf{c})\quad (j=1,\ldots,n), \\
	\mathbf{y}(\tau_{i,n+1}(\mathbf{c})) &= \frac{(1-q\,a_{n+1}\,c)\prod_{l=1}^{n}(1-a_l)}{(1-a_i)\prod_{l=1}^{n}(1-q\,b_l)}\,M_{\tau_{i,n+1}}\,\mathbf{y}(\mathbf{c})\quad (i=1,\ldots,n), \\
	\mathbf{y}(\tau_{n+1,n+1}(\mathbf{c})) &= \frac{\prod_{l=1}^{n}(1-a_l)}{\prod_{l=1}^{n}(1-q\,b_l)}\,M_{\tau_{n+1,n+1}}\,\mathbf{y}(\mathbf{c}).
\end{align*}
\end{cor}

We can show this corollary by a direct calculation or formally by using a relation between two vectors
\[
	\mathbf{y} = \frac{(q\,c;q)_{\infty}\prod_{l=1}^{n}(q\,b_l;q)_{\infty}}{(q\,a_{n+1}\,c;q)_{\infty}\prod_{l=1}^{n}(a_l;q)_{\infty}}\,\mathbf{x},
\]
which follows from the transformation formula between ${}_{n+1}\phi_n$ and $\phi^{(n)}_D$.

\begin{rem}\label{Rem:q-Gar}
If we regard the translations $\tau_{i,j}$ $(i,j=1,\ldots,n+1)$ as elements of $\widetilde{G}$, then they coincide with the original $q$-Garnier system given in \cite{Sak1}.
On the other hand, the original $q$-Garnier system admits a particular solution in terms of $\phi^{(n)}_D$; see \cite{NY1,Sak2}.
These facts are consistent with Corollary \ref{Cor:q-HGE_Lau}.
\end{rem}

\subsection{Proof of Theorem \ref{Thm:q-HGE_Lau}}\label{SubSec:q-HGE_Lau}

Since system \eqref{Eq:q-HGE_Lau1} follows from Theorem \ref{Thm:q-HGE}, we prove system \eqref{Eq:q-HGE_Lau2} here.
For a sake of convenience, we set 
\begin{align*}
	&\widetilde{M}_{p'_0} = c^{-\log_qb_1}M_{p'_0},\quad
	\widetilde{M}_{p'_i} = M_{p'_i}\quad (i=1,\ldots,n-1),\quad
	\widetilde{M}_{p'_n} = c^{-\log_qq^{-1}b_n}M_{p'_n}, \\
	&\widetilde{M}_{\pi_1} = c^{-\log_qq^{-1}b_n}M_{\pi_1}.
\end{align*}
Note that $\widetilde{M}_{\pi_1}$ is invariant under the actions of $p'_0,\ldots,p'_n,\pi_1$.
Then Lemma \ref{Lem:q-HGE_Lau_Thm} implies
\begin{equation}\label{Eq:q-HGE_Lau_Thm1}
	\mathbf{x}(p'_n\,\pi_1(\mathbf{c})) = \frac{1-a_0}{1-b_n}\,\widetilde{M}_{\pi_1}\,\widetilde{M}_{p'_n}\,\mathbf{x}(\mathbf{c}).
\end{equation}
Moreover, we obtain
\begin{align}
	&\widetilde{M}_{p'_i}^{-1} = p'_i(\widetilde{M}_{p'_i})\quad (i=0,\ldots,n), \label{Eq:q-HGE_Lau_Thm2} \\
	&\widetilde{M}_{\pi_1}^{n+1} = q\,c\,I, \label{Eq:q-HGE_Lau_Thm3} \\
	&\widetilde{M}_{\pi_1}\,\widetilde{M}_{p'_i} = \pi_1(\widetilde{M}_{p'_{i+1}})\,\widetilde{M}_{\pi_1}\quad (i=0,\ldots,n), \label{Eq:q-HGE_Lau_Thm4}
\end{align}
where $I$ is the identity matrix, by a direct calculation.

For $j=1,\ldots,n$, system \eqref{Eq:q-HGE_Lau2} follows from Theorem \ref{Thm:q-HGE} and equation \eqref{Eq:q-HGE_Lau_Thm1}.
For $j=n+1$, system \eqref{Eq:q-HGE_Lau2} is derived as follows.
The fundamental relations for $\widetilde{F}$ imply
\begin{align*}
	(\pi_1^{-1}p'_n)^l\,\pi_1^{l+1} &= (\pi_1^{-1}p'_n)^{l-1}\,\pi_1^{-1}p'_n\,\pi_1^{l+1} \\
	&= (\pi_1^{-1}p'_n)^{l-1}\,\pi_1^l\,p'_l \\
	&= \ldots \\
	&= (\pi_1^{-1}p'_n)\,\pi_1^2\,p'_2\ldots p'_l \\
	&= \pi_1\,p'_1\,p'_2\ldots p'_l,
\end{align*}
for $l=1,\ldots,n$.
Equations \eqref{Eq:q-HGE_Lau_Thm1} and \eqref{Eq:q-HGE_Lau_Thm2} imply
\[
	\mathbf{x}(\pi_1^{-1}p'_n(\mathbf{c})) = \frac{1-q\,b_1}{1-a_1}\,\pi_1^{-1}(\widetilde{M}_{p'_n})\,\widetilde{M}_{\pi_1}^{-1}\,\mathbf{x}(\mathbf{c}).
\]
Equation \eqref{Eq:q-HGE_Lau_Thm4} implies
\[
	\widetilde{M}_{\pi_1}^{k+1}\,(\pi_1^{-1}p'_n)^{k-1}\pi_1^{-1}(\widetilde{M}_{p'_n})\,\widetilde{M}_{\pi_1}^{-1} = (\pi_1^{-1}p'_n)^{k-1}\pi_1^k(\widetilde{M}_{p'_k})\,\widetilde{M}_{\pi_1}^k\quad (k=1,\ldots,n).
\]
Combining them and equation \eqref{Eq:q-HGE_Lau_Thm3}, we obtain
\begin{align*}
	&\mathbf{x}(\pi_1\,p'_1\,p'_2\ldots p'_n(\mathbf{c})) \\
	&= \mathbf{x}((\pi_1^{-1}p'_n)^n(\mathbf{c})) \\
	&= \prod_{l=1}^{n}\frac{1-q\,b_l}{1-a_l}\,(\pi_1^{-1}p'_n)^{n-1}\pi_1^{-1}(\widetilde{M}_{p'_n})\,\widetilde{M}_{\pi_1}^{-1}\ldots(\pi_1^{-1}p'_n)\,\pi_1^{-1}(\widetilde{M}_{p'_n})\,\widetilde{M}_{\pi_1}^{-1}\,\pi_1^{-1}(\widetilde{M}_{p'_n})\,\widetilde{M}_{\pi_1}^{-1}\,\mathbf{x}(\mathbf{c}) \\
	&= q^{-1}c^{-1}\prod_{l=1}^{n}\frac{1-q\,b_l}{1-a_l}\,\widetilde{M}_{\pi_1}^{n+1}\,(\pi_1^{-1}p'_n)^{n-1}\pi_1^{-1}(\widetilde{M}_{p'_n})\,\widetilde{M}_{\pi_1}^{-1}\ldots(\pi_1^{-1}p'_n)\,\pi_1^{-1}(\widetilde{M}_{p'_n})\,\widetilde{M}_{\pi_1}^{-1}\,\pi_1^{-1}(\widetilde{M}_{p'_n})\,\widetilde{M}_{\pi_1}^{-1}\,\mathbf{x}(\mathbf{c}) \\
	&= q^{-1}c^{-1}\prod_{l=1}^{n}\frac{1-q\,b_l}{1-a_l}\,(\pi_1^{-1}p'_n)^{n-1}\pi_1^n(\widetilde{M}_{p'_n})\ldots(\pi_1^{-1}p'_n)\,\pi_1^2(\widetilde{M}_{p'_2})\,\pi_1(\widetilde{M}_{p'_1})\,\widetilde{M}_{\pi_1}\,\mathbf{x}(\mathbf{c}) \\
	&= q^{-1}c^{-1}\prod_{l=1}^{n}\frac{1-q\,b_l}{1-a_l}\,\pi_1\,p'_1\ldots p'_{n-1}(\widetilde{M}_{p'_n})\ldots\pi_1\,p'_1(\widetilde{M}_{p'_2})\,\pi_1(\widetilde{M}_{p'_1})\,\widetilde{M}_{\pi_1}\,\mathbf{x}(\mathbf{c}).
\end{align*}
It follows that
\begin{equation}\label{Eq:q-HGE_Lau_Thm5}
	\mathbf{x}(\pi_1\,p'_1\,p'_2\ldots p'_n(\mathbf{c})) = q^{-1}\prod_{l=1}^{n}\frac{1-q\,b_l}{1-a_l}\,\pi_1\,p'_1\ldots p'_{n-1}(M_{p'_n})\ldots\pi_1\,p'_1(M_{p'_2})\,\pi_1(M_{p'_1})\,M_{\pi_1}\,\mathbf{x}(\mathbf{c}).
\end{equation}
Then we obtain system \eqref{Eq:q-HGE_Lau2} from equation \eqref{Eq:q-HGE_Lau_Thm5} and Theorem \ref{Thm:q-HGE}.

\section*{Acknowledgement}

This work was supported by JSPS KAKENHI Grant Number 20K03645.


\end{document}